\documentclass[12pt]{article}
\usepackage{amssymb,amsfonts,amsmath,amsthm, breqn, color, float, mathtools, url, tikz, algorithmic}

\usetikzlibrary{arrows}

\usetikzlibrary{automata,positioning}

\usepackage{caption}
\newcommand{\pp}{{\mathbb P}}

\newcommand{\nn}{{\mathbb N}}

\newcommand{\cale}{{\mathcal E}}

\newcommand{\cals}{{\mathcal S}}

\newcommand{\calo}{{\mathcal O}}

\newcommand{\calv}{{\mathcal V}}

\newcommand{\beq}{\begin{eqnarray*}}
	\newcommand{\feq}{\end{eqnarray*}}
\newcommand{\beqn}{\begin{eqnarray}}
	\newcommand{\feqn}{\end{eqnarray}}

\newtheorem{theorem}{Theorem}
\newtheorem*{conj*}{Conjecture}
\makeatletter \@addtoreset{theorem}{section}\makeatother

\newtheorem{lemma}[theorem]{Lemma}

\newtheorem*{theorema*}{Theorem~A}
\newtheorem*{theoremb*}{Theorem~B}
\newtheorem*{theoremc*}{Theorem~C}
\newtheorem*{theoremd*}{Theorem~D}
\newtheorem*{theoreme*}{Theorem~E}
\newtheorem*{theoremf*}{Theorem~F}
\newtheorem*{cld*}{Condition $\mbox{LD}_d$}
\newtheorem*{theorem*}{Theorem}

\newtheorem{corollary}[theorem]{Corollary}

\usepackage{graphicx}
\def\BState{\State\hskip-\ALG@thistlm}
\makeatother

\DeclareCaptionLabelFormat{noname}{#2}
\newlength\myindent
\DeclareMathOperator{\PP}{\mathbb P}

\DeclareMathOperator{\occ}{\mbox{occ}}

\title{Pattern occurrences in $k$-ary words revisited: a few new and old observations}

\author{Toufik~Mansour\thanks{ Department of Mathematics, University of Haifa, 199 Abba Khoushy Ave, 3498838 Haifa, Israel;
		\newline e-mail: tmansour@univ.haifa.ac.il}
	\and
	Reza~Rastegar\thanks{Occidental Petroleum Corporation, Houston, TX 77046 and Departments of Mathematics and Engineering, University of Tulsa, OK 74104, USA - Adjunct Professor; e-mail:  reza\_rastegar2@oxy.com}}

\begin{document}
	\maketitle
	\begin{abstract}
		In this paper, we study the pattern occurrence in $k$-ary words. We prove an explicit upper bound on the number of $k$-ary words avoiding any given pattern using a random walk argument. Additionally, we reproduce one already known result on the exponential rate of growth of pattern occurrence in words and establish a simple connection among pattern occurrences in permutations and $k$-ary words. A simple yet interesting consequence of this connection is that the Wilf-equivalence of two patterns in words implies their Wilf-equivalence in permutations.
	\end{abstract}
	
	\section{Introduction}
	
	The enumeration of pattern occurrences in discrete sequences has been a very active area of research in the last three decades; see for instance the monographs \cite{Bbook, Kbook}. Although, in the beginning, the main focus was mainly on pattern occurrence in permutations, other classes of sequences including $k$-ary words \cite{noga1, BM, regev1} has been also the subject of further research. The main theme of this paper is also the enumeration of pattern occurrences in $k$-ary words. 
	
	\par
	To state our results we first define a few notations and concepts. Let $\nn:=\{1,2,3,\ldots \}$ and $\nn_0$ denote, respectively, the set of natural numbers and the set of non-negative integers; that is $\nn_0=\nn \cup \{0\}.$ For a given set $A,$ $\#A$ is the cardinality of $A.$ A word $w$ is a sequence $w:=w_1\cdots w_n,$ where $w_i$ is the $i$-the entry of $w$ and is chosen from an arbitrary set of letters also known as alphabet. For any given $k\in \nn,$ we denote the set $\{1,2, \cdots k\}$ by $[k]$. We adopt the convention that $[k]^0=\{\epsilon\},$ where $\epsilon$ is an empty word. A \emph{$k$-ary word} of length $n$ is an element of $[k]^n,$ $n\in\nn.$ A \emph{pattern} is any distinguished word chosen from $\cup_\ell [k]^\ell$ containing letters in $[k]$. Let us now fix integers $k\in \nn,$ $\ell\geq 2,$ and a pattern $v$ in $[k]^\ell$. These parameters are considered to be given and fixed throughout the rest of the paper. An important characteristic of the pattern is its number of distinct letters. We will denote this by $d.$ For instance, if $v=33415,$ then $\ell=5$ and $d=4.$ For an arbitrary word $w$ with length $n\geq \ell,$ an occurrence of the pattern $v$ in $w$ is a sequence of $\ell$ indexes
	$1\leq j_1<j_2<\dots<j_\ell\leq n $ such that the \emph{subsequence} $w_{j_1}\cdots w_{j_\ell}$
	is \emph{order-isomorphic} to the word $v,$ that is
	\beq
	w_{j_p}<w_{j_q}\Longleftrightarrow  v_p<v_q\qquad \forall\,1\leq p,q\leq \ell
	\feq
	and
	\beq
	w_{j_p}=w_{j_q}\Longleftrightarrow  v_p=v_q\qquad \forall\,1\leq p,q\leq \ell.
	\feq
	For any arbitrary word $w$ we denote by $\occ_v(w)$ the number of occurrences of $v$ in $w.$ For instance, if $v$ is the \emph{inversion} $21$
	and $w=35239$, then there are three occurrences, those are $w_1w_3=32,$ $w_2w_3=52,$ and $w_2w_4=53$, and consequently $\occ_v(w)=3$. 	
	We say that a word $w$ \emph{contains the pattern} $v$ exactly $r$ times, $r\in\nn_0,$ if $\occ_v(w)=r.$  By the {\it occurrence subsequence} of $w$, we refer to the minimal length subsequence of the word $w$ containing all the $r$ occurrences of $v$. For instance, in the previous example, the occurrence subsequence are the subsequence $3523.$ For a set of words $A$ and $r\in \nn_0,$ we denote by $F_r^v(A)$ the set of words in $A$ each of which contains $v$ exactly $r$ times. That is,
	\beqn
	\label{gfr}
	F_r^v( A)=\{w\in A:\occ_v(w)=r\}.
	\feqn
	We use $f_r^v(A)$ to refer to $\#F_r^v( A)$. In the case of $r=0,$ we use the term {\it avoidance} instead of {\it occurrence}.
	
	Our first theorem states \\
	
	\begin{theorem} \label{BM_theorem} For any pattern $v$ of length $\ell$ with $d$ distinct letters and any $k\in \nn$ with $k > d$, we have
		\begin{enumerate}
			\item[(a)] $(d-1)^n \leq f_0^v([k]^n) \leq (d-1)^n \sum_{i=0}^{\ell \binom{k}{d}} \binom{n}{i} \left(\frac{k-d+1}{d-1}\right)^i, $ whenever $n> \ell \binom{k}{d}$.		
			\item[(b)] $\lim_{n\to \infty} f_r^v([k]^n)^{\frac{1}{n}} = d-1,$ whenever $r\in \nn_0.$
			\item[(c)] $f_0^v([k]^n)=\sum_{i=0}^{n} a_{i,k,v}(n)(d-1)^i,$
			where $a_{i,k, v}(n)$ are polynomials in $n$ with integer coefficients possibly depending on $k$ and $v$.     		
		\end{enumerate}
	\end{theorem}
	Part (a) gives a new general upper bound on $f_0^v([k]^n)$ which is a function of $\ell$, $d$, $k$, and $n$. Its proof, given in Section \ref{kn-sec}, is based on a simple random walk argument over the poset of instances of the pattern $v$ for the alphabet $[k]$. This poset is defined in Section \ref{pattern-sec}.  
	
	The statement of part (b) is not new, however, our proof is. The existence of this limit and its value was first proven for the $r=0$ case, i.e. pattern avoidance, in Theorem 3.2 \cite{BM} using  transfer matrix method (see section 4.7 of \cite{stan1}). The result was extended later to the general occurrence case $r\in \nn$ in Theorem 2.7 \cite{MRR} by a similar technique. Our approach for $r=0$ case is similar to that of \cite{BM}, in the sense that, the number of walks in certain graphs are counted. The advantages of our method in proving the case $r=0$ are however twofold: (1) the poset defined in Section \ref{pattern-sec} is simpler to construct and more intuitive than those defined in \cite{BM, MRR}. (2) our approach has a generalizable probabilistic flavor; for instance, our results can be extended to other contexts, such as when the words are generated by any finite irreducible aperiodic Markov chain. Our proof of the $r\in \nn$ case is entirely different from the transfer matrix method used in \cite{MRR} and it is done directly by relating the $r\in \nn$ case to the $r=0$ case. \\
	\par
	Recall that a word $w=w_1\cdots w_n$ is a permutation if it is simply a re-arrangement of the sequence $12\ldots n$; that is $w\in F_0^{11}([n]^n).$ We use $\cals_n$ to refer to the set of all permutations of length $n$.
	The celebrated Marcus-Tardos Theorem \cite{MT}, confirming Stanley-Wilf conjecture for pattern avoidance in $\cals_n$, states that the number of permutations
	in $\cals_n$ avoiding a permutation pattern $v\in \cals_d$ grows at most as fast as an exponential function of $n$; that is, there exists a constant $c_v$ where for all $n\in \nn$,
	\beq
	f_0^v(\cals_n) \leq c_v^n.
	\feq	
	Around the time this conjecture was settled, Br\"and\'{e}n and Mansour \cite{BM} conjectured the plausibility of equivalence of Marcus-Tardos Theorem for $\cals_n$ with a similar result for $[n]^n$. More precisely, they conjectured that there are $c_v$ and $g_v$ where for all $n\in \nn$ the following statement holds:
	\beqn \label{BM_conj}
	f_0^v(\cals_n) < c_v^n \quad \text{if and only if} \quad f_0^v([n]^n) < g_v^n.
	\feqn
	By a reformulation of the pattern avoidance in the language of graph theory, Corollary 2.2 \cite{klazar4} settled the existence of $g_v>0$ where $f_0^v([n]^n) < g_v^n$ for all $n\in \nn$. However, yet there is no direct explanation on why \eqref{BM_conj} actually holds and whether something nontrivial could be said about the relationship between the optimal values of $c_v$ and $g_v$. That being said, our motivation behind Theorem \ref{i^nSn} bellow was originally to establish some connection between $f_0^v([n]^n)$ and $f_0^v(\cals_0^v)$ (see also the inequality \eqref{sdisc}):	
	
	\begin{theorem} \label{i^nSn} Let $v$ be any pattern. Then, for $r\in \nn_0,$
		\beqn \label{f_perm_n}
		&&f_r^v(\cals_n) =  \sum_{k=1}^{n} (-1)^{n-k} \binom{n}{k} f_r^v([k]^n).
		\feqn
		Furthermore, let $S_r^v(x)$ and $W_{r,k}^v(x)$ be the exponential generating functions of $f_r^v(\cals_n)$ and $f_r^v([k]^{n})$; those are,
		\beq
		&&S_r^v(x) = \sum_{n=1}^\infty \frac{f_r^v(\cals_n)}{n!}x^n \quad \mbox{and}\quad W_{r,k}^v(x) = \sum_{n=1}^\infty \frac{f_r^v([k]^{n})}{n!}x^n, \quad k\in \nn.
		\feq
		Then	
		\beqn \label{exp_relations}
		&&S_r^v(x) = \sum_{k=1}^\infty \frac{(-x)^k}{k!} \frac{\partial^k}{\partial x^k} W_{r,k}^v(-x).
		\feqn
	\end{theorem} 	
	We remark that \eqref{f_perm_n} is indeed correct if we replace the pattern $v$ with a set of patterns. We end this section by stating an interesting immediate consequence of Theorem~\ref{i^nSn}:
	
	\begin{corollary}
		For any two patterns $v_1$ and $v_2$, if $f_0^{v_1}([k]^n)=f_0^{v_2}([k]^n)$ for all $k$ and $n$, then $f_0^{v_1}(\cals_n)=f_0^{v_2}(\cals_n)$ for all $n$.  The inverse statement does not hold. 	
	\end{corollary}
	
	This implies that Wilf-equivalence in words implies Wilf-equivalence in permutations. An example for the ``if'' statement is the following: it was shown in \cite{Burstein} that  
	\beq
	f_0^{123}([k]^n) = f_0^{132}([k]^n) =\delta_{k,1}+2^{n-2(k-2)}\sum_{j=0}^{k-2}\binom{n+2j}{n} \left(\sum_{m=j}^{k-2}\frac{1}{m+1}\binom{2m}{m}\binom{2k-2m}{k-m} \right),
	\feq 
	where $\delta_{k,1}=1$ if $k=1$, and is zero otherwise. Therefore,  the corollary implies $f_0^{123}(\cals_n)=f_0^{132}(\cals_n)$. This is indeed known to hold by a direct calculation \cite{Bbook}, showing 
	\beq
	f_0^{123}(\cals_n)=f_0^{132}(\cals_n) = \frac{1}{n+1}\binom{2n}{n}.
	\feq 
	To see the inverse statement in the corollary, take for example $v_1=1324$ and $v_2=2413$. As it is shown by B\'ona \cite{bona3} $f_0^{v_1}(S_n)=f_0^{v_2}(S_n)$; however, we know from \cite{tj1}, $f_0^{v_1}([k]^n)$ and $f_0^{v_2}([k]^n)$ are not same for all $k$ and $n$.

	\section{Instances of a pattern} \label{pattern-sec}
	
	In this section, we first define a partially ordered set (poset) based on the instances of a given pattern $v$ in the alphabet $[k]$ and discuss a few of its basic properties. To that end, for a given pattern $v$ of length $\ell$ with $d$ distinct letters, we define $\Phi_{k}(v)$ to be the set of all words $w\in[k]^{\ell}$ for which $\occ_v(w)=1.$ The size of this set, $\phi_k(v):=\#\Phi_k(v)$, is $\binom{k}{d}.$ For instance,
	\beq
	\Phi_4(123) = \{ 123,124,134,234 \}.
	\feq
	To decide whether a given word $w$ is in $F_0^v([k]^n)$ or not, we can advise a simple {\it avoidance detection algorithm}, as follows: it reads the letters in $w$ from left to right and updates a state vector reflecting the prefixes of the pattern instances in $\Phi_k(v)$ that have been encountered after any letter of $w$ is consumed. Clearly, the word $w$ avoids the pattern $v$, as long as, none of the state vector entries coincides with any element in $\Phi_k(v)$ by the time that the rightmost letter of $w$ is processed. Intuitively, this is a walk on the space of ``feasible" state vectors. To formalize this idea, we equip $\Phi_k(v)$ with an order $<_\circ$: for any pair $u,t\in \Phi_k(v)$, we say $u<_\circ t,$ if and only if, $u_i<_\circ t_i$ for the first index $1\leq i\leq \ell$ where $u_i\neq t_i$. For example, elements in $\Phi_4(123)$ are ordered as
	\beqn \label{instance_order_exp}
	123 <_\circ 124 <_\circ 134 <_\circ 234.
	\feqn
	For each $u\in \Phi_k(v)$, set $\Xi(u)$ to be the set of all prefixes of $u$ including the empty word $\epsilon$ and $u$ itself. For instance,
	\beq
	\Xi(1245) = \{ \epsilon, 1, 12, 124, 1245 \}.
	\feq
	Regardless of the value of $u\in \Phi_k(v),$ the size of $\Xi(u)$ is always $\ell+1$. 
	Next, set $\xi^\epsilon:=(\epsilon,\epsilon,\cdots,\epsilon)$ to be the null state vector and define
	\beq
	\Xi_k(v) := \{\xi:=(\xi_u)_{u\in \Phi_k(v)} \ | \ \xi_u \in \Xi(u),\  \forall u\in \Phi_k(v) \}.
	\feq
	The indexes of the entries of any $\xi\in \Xi_k(v)$s are ordered according to $<_\circ$. \\
	\par
	For any two given states $\xi, \eta\in \Xi_k(v)$, we say $\xi$ is extendable to $\eta$ in one step if there are a non-empty subset $A\in \Phi_k(v)$ and a unique $i\in [k]$ for which $\eta_t = \xi_t i$ (concatenation of the word $\xi_t$ and letter $i$) for $t\in A,$ and $\eta_t = \xi_t$, otherwise. We use the notation $\nu = \xi i$ in this case. Additionally, we define an order $<_\diamond$ on $\Xi_k(v)$: for any pair of  $\eta,\xi\in \Xi_k(v),$ we say $\xi <_\diamond \eta$, if and only if, $\xi$ is extendable to $\eta$ in $s$ steps for some finite $s\in\nn$; that is, $\eta=\xi i_1\dots i_s:=(\ldots ((\xi i_1)\ldots )i_s$ for a finite sequence of letters $i_1,\cdots, i_s\in [k]$. Let $\cale_k(v)$ be the set
	\beq
	\cale_k(v):=\{ \xi^\epsilon \} \cup \{ \xi\in \Xi_k(v), \xi^\epsilon<_\diamond \xi \}.
	\feq
	Then, each $\xi\in \cale_k(v)$ can be obtained by a finite extension of $\xi^\epsilon$; that is, $\xi^\epsilon$ is extendable in some finite number of steps to $\xi.$  Define
	\beq
	\calo_k(v) := \{ \xi \in \cale_k(v) \ | \  \xi_u=u \mbox{ for some } u\in \Phi_k(v) \}
	\feq
	and set $\calv_k(v):=\cale_k(v)\setminus \calo_k(v)$. It is clear that $\calv_k(v)$ is a poset with respect to $<_\diamond$. See Figures \ref{fig1}-\ref{fig3} for Hasse diagram of several examples of this poset with different patterns and the alphabet $[4]=\{1,2,3,4\}$. For each $\xi\in\calv_k(v)$, let $L_{v,k}(\xi)$ be the subset of alphabet $[k]$ whose elements extend $\xi$ in $\calv_k(v)\cup\calo_k(v).$
	
	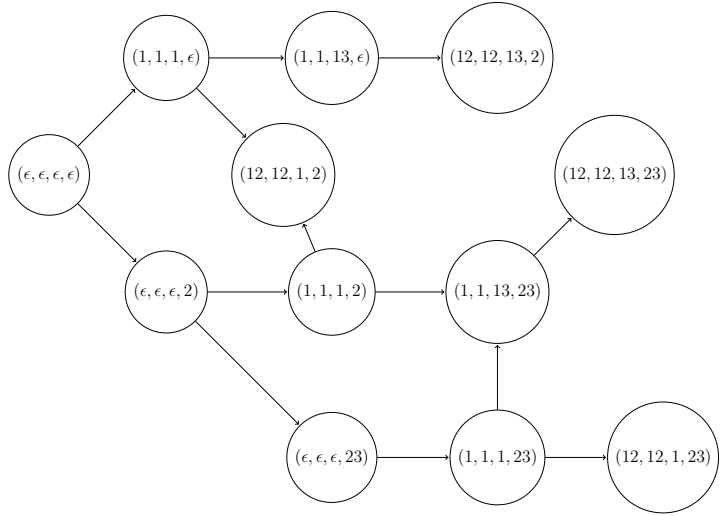
\begin{figure} [htp]
		\centering	
		\resizebox{0.7\textwidth}{!}{
			\begin{tikzpicture}[shorten >= 1pt,node distance=4cm,on grid,auto]
				\node[state] (q_0)   {
					$(\epsilon,\epsilon,\epsilon,\epsilon)$};
				\node[state] (q_1) [above right=of q_0] { $
					(1,1,1,\epsilon)$};
				\node[state] (q_2) [below right=of q_0] { $(\epsilon, \epsilon, \epsilon, 2)$};
				\node[state](q_3) [below right=of q_1] {$(12,12,1,2)$};
				\node[state](q_4) [right =of q_1] {$(1,1,13,\epsilon)$};
				\node[state](q_12) [right =of q_4] {$(12,12,13,2)$};		
				\node[state](q_5) [right =of q_2] {$(1,1,1,2)$};
				\node[state](q_6) [below =of q_5] {$(\epsilon,\epsilon, \epsilon,23)$};
				\node[state](q_8) [right=of q_5] {$(1,1,13,23)$};
				\node[state](q_13) [above right=of q_8] {$(12,12,13,23)$};		
				\node[state](q_9) [right=of q_6] {$(1,1,1,23)$};		
				\node[state](q_11) [right=of q_9] {$(12,12,1,23)$};							
				\path[->]
				(q_0)  edge  (q_1) edge  (q_2)
				(q_1)  edge  (q_3) edge  (q_4)	
				(q_4) edge (q_12)
				(q_2) edge  (q_5) edge  (q_6)	
				(q_5)  edge  (q_3) edge (q_8)
				(q_6)  edge  (q_9)
				(q_8)  edge  (q_13)
				(q_9)  edge  (q_8) 		
				(q_9)  edge  (q_11);
			\end{tikzpicture}
		}
		\caption{The Hasse diagram for the poset $\calv_k(v)$ with $v=123$ and $k=4.$} \label{fig1}
	\end{figure}
	
	\begin{figure} [htp]
		\centering	
		\resizebox{0.7\textwidth}{!}{
			\begin{tikzpicture}[shorten >= 1pt,node distance=4cm,on grid,auto]
				\node[state] (q_0)   {
					$(\epsilon,\epsilon,\epsilon,\epsilon)$};
				\node[state] (q_1) [above right=of q_0] { $(2,2,\epsilon,\epsilon)$};
				\node[state] (q_2) [below right=of q_0] { $(\epsilon, \epsilon, 3, 3)$};
				\node[state](q_3) [above right=of q_1] {$(21,21,\epsilon,\epsilon)$};
				\node[state](q_4) [below =of q_3] {$(2,2,3,3)$};
				\node[state](q_5) [right =of q_4] {$(21,21,31,3)$};		
				\node[state](q_7) [below right=of q_2] {$(\epsilon,\epsilon, 31,3)$};
				\node[state](q_8) [right=of q_2] {$(2,2,3,32)$};
				\node[state](q_9) [right=of q_7] {$(2,2,31,32)$};
				\node[state](q_10) [right=of q_9] {$(21,21,31,32)$};		
				\path[->]
				(q_0)  edge  (q_1) edge  (q_2)
				(q_1)  edge  (q_3) edge  (q_4)	
				(q_4) edge (q_5) edge (q_8)
				(q_5) edge (q_10)	
				(q_2) edge  (q_7) edge  (q_8)
				(q_7) edge (q_9)
				(q_9) edge (q_10)					
				(q_8)  edge  (q_10);
			\end{tikzpicture}
		}
		\caption{The Hasse diagram for the poset $\calv_k(v)$ with $v=213$ and $k=4.$} \label{fig2}
	\end{figure}
	
	\begin{figure} [htp]
		\centering	
		\resizebox{0.7\textwidth}{!}{
			\begin{tikzpicture}[shorten >= 1pt,node distance=4cm,on grid,auto]
				\node[state] (q_0)   {$(\epsilon,\epsilon,\epsilon,\epsilon)$}; 
				\node[state] (q_1) [above right=of q_0] { $(1,1,1,\epsilon)$}; 
				\node[state] (q_2) [below right=of q_0] { $(\epsilon, \epsilon, \epsilon, 2)$}; 
				\node[state] (q_4) [below right=of q_1] { $(1,1, 1,2)$};
				\node[state] (q_3) [right=of q_1] { $(13,1,1,\epsilon)$};
				\node[state] (q_7) [right=of q_3] { $(13,14,14,\epsilon)$};
				\node[state] (q_4) [below right=of q_1] { $(1,1, 1,2)$};
				\node[state] (q_5) [above right=of q_1] { $(1,14, 14,\epsilon)$};
				\node[state] (q_8) [right=of q_4] { $(13,1,1,2)$};			
				\node[state] (q_6) [below right=of q_2] { $(\epsilon,\epsilon,\epsilon,24)$};
				\node[state] (q_10) [right=of q_6] { $(1,1,1,24)$};
				\node[state] (q_11) [ right=of q_10] { $(1,14,14,24)$};
				\node[state] (q_12) [right=of q_8] { $(13,14,14,24)$};\					 			
				\path[->] 
				(q_0)  edge  (q_1) edge  (q_2)
				(q_1)  edge  (q_3) edge  (q_4) edge (q_5)
				(q_2)  edge  (q_4) edge  (q_6)
				(q_3)  edge  (q_7)			
				(q_4)  edge  (q_8) 
				(q_6)  edge  (q_10)
				(q_8)  edge (q_12)
				(q_10)  edge  (q_11);								
			\end{tikzpicture}
		}
		\caption{The Hasse diagram for the poset $\calv_k(v)$ with $v=132$ and $k=4.$} \label{fig3}
	\end{figure}

	\begin{figure} [!t]
		\centering	
		\begin{algorithmic}
			\STATE Take $v$, $w$ as the input data.
			\STATE Let $n$ be the length of $w$, and $k$ be the number of distinct letters in $w$. 
			\STATE Set $\xi:=\xi^\epsilon$ and $i=1$.
			\STATE Compute $\Phi_k(v),$ $\cale_k(v)$, $\calv_k(v)$, and  $L_{v,k}(.)$.
			\WHILE{$i \leq n$}
			\IF {$w_i\in L_{v,k}(\xi)$}
			\STATE Set $\xi$ to $\xi w_i$
			\ENDIF
			\STATE Increase $i$ by one.
			\ENDWHILE
			\IF {$\xi \in \calv_k(v)$}
			\RETURN $w$ avoids $v$
			\ELSE
			\RETURN $w$ includes $v$
			\ENDIF
		\end{algorithmic}
		\caption{The pseudo-code of the avoidance detection algorithm.} \label{alg1}
	\end{figure}	
	
	With these notations in hand, the avoidance detection algorithm described informally in the beginning of this section is formalized as shown in Figure \ref{alg1}. In the rest of this section, we provide some basic information on two important objects used by this algorithm; namely $\calv_k(v)$ and $L_{v,k}(.)$. The first result gives an upper bound on the depth of the poset $\calv_k(v)$:
	
	\begin{lemma} \label{chain-lemma}
		For any pattern $v$, let $h_k(v)$ be the size of the maximum chain of $\calv_k(v).$ Then,
		\beqn \label{upper_depth}
		h_k(v) \leq \binom{k}{d}\ell.
		\feqn
	\end{lemma}
	\begin{proof}
		Recall that every element of $\calv_k(v)$ has $\binom{k}{d}$ entries, where for each $u\in \Phi_k(v)$, the corresponding entry can only take its values from the set of $\ell$ distinct values 
		\beq
		\{\epsilon, u_1, u_1u_2, \cdots, u_1\ldots u_{\ell-1}\}.
		\feq
		Hence the result follows.
	\end{proof}
	
	We end this section with the following observation.
	
	\begin{lemma} \label{li_lemma}
		Suppose $v$ is any pattern with $d$ distinct letters. For any $\xi\in \calv_k(v),$ we have
		\beqn
		\# L_{v,k}(\xi) \geq k-d+1. \label{card_ui2}
		\feqn	
	\end{lemma}
	\begin{proof}
		The basic idea of the proof is that with the alphabet $[k]$ and the pattern $v$ made of $d$ distinct letters, we can bound $\#L_{v,k}(.)$ from bellow in terms of either $\#L_{v',k-1}(.)$ or $\#L_{v,k-1}(.)$ for a certain pattern $v'$ with $d-1$ distinct letters. Hence, an inductive argument with respect to $k$ and $d$ would be a natural option to prove the statement. First, note that for any pattern $v$ with $d$ distinct letters, we have 
		\beqn
		\#L_{v,d}(\xi) &=& 1, \mbox{ for all } \xi\in \calv_d(v),\notag \\
		L_{v,k}(\xi^\epsilon) &=& \{v_{1},\cdots, k - d + v_{1} \}, \quad \#L_{v,k}(\xi^\epsilon) = k-d+1. \label{card_ui}
		\feqn   
		Furthermore, for $v$ with $d=1$ 
		\beq 
		\#L_{v,k}(\xi) = k \mbox{ for all } \xi \in \calv_k(v).
		\feq 
		Now, we state our induction hypothesis: 
		\begin{enumerate}
			\item[] For some fixed $d, k\in \nn$ with $k\geq d$, and any pattern $v$ with $d_0\leq d$ distinct letters, if $\xi \in \calv_k(v),$ then we have $\#L_{v,k_0}(\xi) \geq k_0-d_0+1$ with $d_0\leq k_0\leq k$. 
		\end{enumerate}
		Finally, we complete the induction argument by showing that for these $k$ and $d$,
		\beqn 
		\#L_{v,k+1}(\rho) &\geq& k-d+2, \mbox{ for all } \rho \in \calv_{k+1}(v), \label{case1}\\
		\#L_{v'',k}(\rho) &\geq& k-d, \mbox{ for all } \rho \in \calv_{k}(v''), \label{case2}
		\feqn
		where $v''$ is a pattern obtained by inserting one or more copies of the letter $d+1$ into any arbitrary position(s) in $v$. 
		\par
		To that end, let $v$ be any pattern of $d$ distinct letters. If $\rho=\rho^\epsilon:=(\epsilon, \cdots, \epsilon)\in \calv_{k+1}(v)$, then by \eqref{card_ui} we are done. Otherwise, we consider two cases for any other $\rho \in \calv_{k+1}(v)$:
		\begin{enumerate}
			\item if $k+1\in L_{v,k+1}(\rho)$, then define $\xi$ to be the set of $\rho_u$s where $u$ does not have the letter $k+1$. Clearly $\xi \in \calv_k(v)$, $k+1\notin L_{v,k}(\xi)$, 
			\beq
			L_{v,k}(\xi) \cup \{k+1\} \subset L_{v,k+1}(\rho), \mbox{ and } \#L_{v,k}(\xi) + 1 \leq \#L_{v,k+1}(\rho).
			\feq 
			Here, $\xi$ should be understood as the vector of elements in the set ordered by $<_\circ.$
			\item if $k+1\notin L_{v,k+1}(\rho)$, then define $v'$ to be a pattern with $d-1$ distinct letters; obtained by dropping the letter $d$ everywhere in $v$. Also, define $\xi$ to be the set of $\nu_u$s after dropping the letter $k+1$ (if any) where $u$ has the letter $k+1$. In this case, we have $\xi \in \calv_k(v')$,
			\beq
			L_{v',k}(\xi) \subset L_{v,k+1}(\rho), \mbox{ and } \#L_{v',k}(\xi) \leq \#L_{v,k+1}(\rho).
			\feq
		\end{enumerate}
		These two cases along with the induction hypothesis yield \eqref{case1}. Inequality \eqref{case2} can also be proven using an identical argument. This completes the proof. 
	\end{proof}

	\section{Proof of Theorem \ref{BM_theorem}} \label{kn-sec}
	This section is devoted to the proof of Theorem~\ref{BM_theorem}. We start with a few definitions.
	\par
	Let $\PP$ be a probability law that induces the uniform distribution on $[k]$. For
	technical convenience, we enlarge the probability space of $\PP$ such that it contains all random variables used in the following. 
	Let $(X(n))_{n\in \nn}$ be an i.i.d. sequence of random variables where $X(1)$ is distributed as a uniform variable on $[k]$; that is,
	\beq
	\pp(X(n)=i)= \frac{1}{k} \quad \forall i\in [k], \ \ n \in \nn.
	\feq
	Set the sequence $Z:=(Z(n))_{n\in \nn_0}$ where $Z(0):=\xi^\epsilon$ and
	\beq
	Z(n+1)=(E_u(Z_u(n),X(n+1)))_{u \in \Phi_v(k)},
	\feq
	where $E_u(.,.): \Xi(u) \times [k] \to  \Xi(u)$ is defined by
	\beq
	E_u(\xi_u, x) =
	\begin{cases}
		\xi_u &\quad\text{if } \quad \xi_u x \notin \Xi(u) \\
		\xi_u x &\quad\text{otherwise,}
	\end{cases}
	\feq
	for $\xi_u \in \Xi_u,$ $x\in [k],$ and $u\in \Phi_k(v)$.
	\par
	Note that $Z$ is a Markov chain on $\cale_k(v)$ and we use $P_{v,k}$, or simply $P$, to refer to the transition matrix of $Z(n)$ defined as follows. For any two given states $\xi, \nu\in \cale_k(v),$
	\beqn \label{pxixi}
	P(\xi,\nu) = P_{v,k}(\xi,\nu) := \begin{cases}
		\frac{1}{k} & \mbox{$\nu=\xi i$ for some $i\in [k]$}, \\
		1-\frac{\# L_{v,k}(\xi) }{k} & \nu=\xi, \\
		0 & \mbox{otherwise.}\\
	\end{cases}
	\feqn
	\par
	By Lemma \ref{li_lemma}, this yields, for $\xi\in \calv_k(\xi)$,
	\beqn \label{Pxixi}
	P(\xi^\epsilon,\xi^\epsilon) = \frac{d-1}{k}, \quad\mbox{and}\quad 	P(\xi,\xi) \leq \frac{d-1}{k}.
	\feqn
	
	To summarize, $P$ may be written as an upper triangular matrix in which the values of all non-zero non-diagonal entries are $1/k$ and the diagonal entries are bounded above by $(d-1)/k$.											
	\par
	We are now ready to give the
	
	\begin{proof}[Proof of Theorem \ref{BM_theorem}-(a) and (c)]
		By the avoidance detection algorithm (Figure \ref{alg1}) and the definition of $Z(n),$ $F_0^v([k]^n)$ can be described in terms of the set of realizations of the Markov chain $Z$ up to the time $n$ in which $Z(0), Z(1), \cdots, Z(n) \in \calv_k(v)$. That being said, we may write
		\beq																		f_0^v([k]^n)= k^n\sum_{\xi \in \calv_k(v) } P^{(n)}(\xi^\epsilon, \xi),
		\feq
		where $P^{(n)}$ is the $n$-fold transition probability matrix of the Markov chain $Z$. Given that all nonzero entries of $P$ are of form $\frac{i}{k},$ with $1\leq i\leq d-1$, and $P(\xi^\epsilon, \xi^\epsilon)=\frac{d-1}{k}$, the proof of the part (c) follows.
		\par
		We now prove the part (a). To that end, the lower bound is  simply obtained by the observation that any word in $[d-1]^n\subset [k]^n$ avoids $v$. To obtain the upper-bound, first we define a pure birth process $(\tilde Z(n))_{n\in \nn_0}\subset \nn^\nn$, where $$\PP(\tilde Z(0)=0)=1$$ and 
		\beq
		\PP(\tilde Z(n+1) = i | \tilde Z(n) = i) = 1- \PP(\tilde Z(n+1) = i+1 | \tilde Z(n) = i) = \frac{d-1}{k},
		\feq
		for all $n\in \nn_0$ and $i\in \nn$.
		Clearly, $\tilde Z(n+1)-\tilde Z(n)$ is a Bernoulli random variable with the success probability $\frac{k-d+1}{k}.$ Hence, $\tilde Z(n)$ is sum of $n$ Bernoulli random variables of parameter $\frac{k-d+1}{k}.$
		\par
		Second, we point out that
		\begin{enumerate}
			\item once the process $Z(n)$ leaves a given state $\xi$, it will never come back to $\xi$ again,
			\item starting from $\xi^\epsilon,$ $Z(n)$ can take maximum $h_k(v)$ distinct values before leaving the set $\calv_k(v)$,
			\item by \eqref{Pxixi}, $Z(n)\in \calv_k(n)$ stays in its current state with maximum probability $(d-1)/k$.
		\end{enumerate}
		Thus,
		\beq
		&& \PP(Z(n) \in \calv_k(v)) \leq \PP(\tilde Z(n) \leq h_k(v)) \\
		&& \quad = \frac{1}{k^n} \sum_{i=0}^{h_k(v)} \binom{n}{i} (k-d+1)^i(d-1)^{n-i} \leq \frac{(d-1)^n}{k^n} \sum_{i=0}^{\ell \binom{k}{d}} \binom{n}{i} \left(\frac{k-d+1}{d-1}\right)^i,
		\feq
		where for the last inequality we used \eqref{upper_depth}. Finally, note that $f_{0}^v([k]^n)$ is simply $k^n\PP(Z(n) \in \calv_k(v))$ and hence the proof is complete.
	\end{proof}
	
	We conclude this section with the
	
	\begin{proof} [Proof of Theorem \ref{BM_theorem}-(b)]
		The $r=0$ case is an immediate consequence of part (a) since the summation term in the upper bound is a finite polynomial of $n$:
		\beq
		d-1 &\leq& \liminf_{n\to \infty} f_0^v([k]^n)^{\frac{1}{n}} \leq \limsup_{n\to \infty} f_0^v([k]^n)^{\frac{1}{n}} \\
		&\leq& (d-1) \limsup_{n\to \infty}  \left( \sum_{i=0}^{\ell \binom{k}{d}} \binom{n}{i} \left(\frac{k-d+1}{d-1}\right)^i \right)^{\frac{1}{n}} = d-1.
		\feq
		\par
		To prove the result for the general case $r\in \nn,$ let $f_{r,s}^v([k]^n)$ be the number of words in $[k]^n$ with exactly $r$ occurrences of the pattern $v$  and the occurrence subsequence of length $s$. Clearly, for $s \notin \{\ell+r-1, \cdots,  r\ell-1, r\ell \}$, we have $f_{r,s}^v([k]^n)=0.$ Hence, $f_{r}^v([k]^n)$ can be written as
		\beqn \label{f_decom1}
		f_{r}^v([k]^n) = \sum_{s=\ell+r-1}^{r\ell} f_{r,s}^v([k]^n).
		\feqn
		Since for each $k$-ary word of length $n$ with $r$ occurrences of $v$ and $s$ occurrence subsequence
		\begin{enumerate}
			\item removing the entire occurrence subsequence gives a new word of length $n-s$ that avoids $v$,
			\item the occurrence subsequence may happen in any of $\binom{n}{s}$ locations in the word, and
			\item the number of possibilities for the occurrence subsequence is bounded above by $k^s$,
		\end{enumerate}
		then
		\beqn \label{f_decom2}
		f_{0}^v([k]^{n-s}) \leq f_{r,s}^v([k]^n) \leq k^s \binom{n}{s}f_{0}^v([k]^{n-s}).
		\feqn
		Equations \eqref{f_decom1} and \eqref{f_decom2} together imply
		\beq
		f_0^v([k]^{n-r\ell}) \leq f_{r}^v([k]^n) =  \sum_{s=\ell+r-1}^{r\ell} f_{r,s}^v([k]^n) \quad \leq \sum_{s=\ell+r-1}^{r\ell } k^s \binom{n}{s} f_{0}^v([k]^{n-s}).
		\feq
		Hence, by taking the limit from all the sides of the inequality, we arrive at
		\beq
		&&\liminf_{n\to \infty} f_{0}^v([k]^{n-r\ell})^{\frac{1}{n}} \leq \liminf_{n\to \infty} f_{r}^v([k]^n)^{\frac{1}{n}} \\
		&&\quad  \leq \limsup_{n\to \infty} f_{r}^v([k]^n)^{\frac{1}{n}} \leq \limsup_{n\to\infty} \left(\sum_{s=\ell+r-1}^{r\ell} k^s \binom{n}{s} f_{0}^v([k]^{n-s}) \right)^{\frac{1}{n}}.
		\feq
		Left side is clearly $d-1$ by the $r=0$ case. The right side is also $d-1$ since
		\begin{enumerate}
			\item it is a finite sum with the number of terms independent of $n$, and
			\item the coefficients of $f_0^v([k]^{n-s})$s are finite degree polynomials in $n$.
		\end{enumerate}
		The proof of the general case $r\in \nn$ is complete.
	\end{proof}
	
	\section{Proof of Theorem \ref{i^nSn}} 	
	
	This section is devoted to the the proof of Theorem \ref{i^nSn} relating the pattern occurrence in $[k]^n$ and in $\cals_n$. \\
	
	\begin{proof} [Proof of Theorem \ref{i^nSn}-(a)]
		Recall \eqref{gfr}. Let $Y_n(A)$ be the set of all words of length $n$ where the distinct set of letters of each word is exactly the set $A$. It is defined as
		\beq
		&&Y_n(A) := A^n \setminus \left(\cup_{e\in A} A_e^n\right)\quad \mbox{where}\quad A_e:=A\setminus\{e\}.
		\feq	
		Observe that this implies
		\beqn  \label{LnA}
		F_r^v(Y_n(A)) = F_r^v(A^n) \setminus \left(\cup_{e\in A} F_r^v(A_e^n)\right).
		\feqn
		Additionally, for any non-empty subset $I\subset  A,$
		\beq
		\cup_{e\in I} F_r^v(A_e^n) = F_r^v(\cup_{e\in I} A_e^n)
		\feq
		and
		\beq
		\cap_{e\in I} F_r^v(A_e^n) = \cap_{e\in I} F_r^v(A_e^n).
		\feq
		Then, an application of the inclusion exclusion principle to \eqref{LnA} yields
		\beq
		&& f_r^v(Y_n(A)) = f_r^v(A^n) - \sum_{I\subset A} (-1)^{n-\#I}f_r^v(\cap_{e\in I} A_e^n ) \\
		&& = \sum_{I\subset A} (-1)^{\#I}f_r^v((A\setminus I)^n ).
		\feq
		Since $f_0^v(A^n) = f_0^v(B^n)$ for any $A,B\subset [n]$ where $A$ and $B$ are the same size, for each $a\in \nn,$ we get	
		\beq
		f_r^v(Y_n([a])) = f_r^v([a]^n) + \sum_{k=1}^{a-1} (-1)^{k} \binom{a}{k} f_r^v([a-k]^n).
		\feq	
		Finally, observe that $Y_n([n]) = \cals_n$. This completes the proof of \eqref{f_perm_n}. \par
		To prove \eqref{exp_relations}, note that
		\beq
		\frac{\partial^k}{\partial x^k}W^v_{r,k}(-x) = (-1)^{k} \sum_{n=k}^\infty \frac{(-x)^{n-k}}{(n-k)!}f_r^v([k]^{n}).
		\feq
		By multiplying the terms in Lemma \ref{i^nSn} by $x^n/n!$, adding up over natural numbers, and some simplification we have
		\beq
		&&S^v_r(x) = \sum_{n=1}^\infty \frac{x^n}{n!} \sum_{k=1}^{n} (-1)^{n-k} \binom{n}{k} f_r^v([k]^n) \\
		&& \qquad = \sum_{k=1}^\infty \frac{x^k}{k!} \sum_{n=k}^{\infty} \frac{(-x)^{n-k}}{(n-k)!} f_r^v([k]^n) = \sum_{k=1}^\infty \frac{(-x)^k}{k!} \frac{\partial^k}{\partial x^k}W^v_{r,k}(-x).
		\feq
		The proof of this part is complete.
	\end{proof}

	It is not hard to see $f_0^{12}([k]^n)$ is given by
	$\binom{n+k-1}{n}$ and $f_0^{12}(\cals_n)=1$. Thus, an application of Theorem \ref{i^nSn}-(a) leads to the following identity
	\beq
	\sum_{k=1}^{n} (-1)^{n-k} \binom{n}{k}\binom{n+k-1}{n} = 1.
	\feq
	
	We also point out that Theorem \ref{BM_theorem}-(c), together with \eqref{f_perm_n}, yields
	\beqn \label{Sn_poly}
	f_r^v(\cals_n) = \sum_{i=0}^{n} (d-1)^i  \left(\sum_{k=1}^{n} (-1)^{n-k} \binom{n}{k}  a_{i,k,v}(n) \right),
	\feqn
	where all $a_{i,k, v}(n)$s are polynomials in $n$ with integer coefficients possibly depending on $k$, $v$, and $d$. Identity \eqref{Sn_poly} may provide some direct informal justification about why $c_v$ is an exponential function of $d$ for certain patterns. This statement was formally established by Fox \cite{fox} in a recent paper, where the proof relies on a refinement and extension of the framework developed in \cite{MT}.
	\par
	
	We finally end this section with the following observation. Note that for any set $A$, the partition $A^n = \cup_{I \subseteq A} Y_n(A\setminus I)$ yields
	\beqn \label{rel1}
	f_0^v(A^n ) = \sum_{I\subseteq A} f_0^v(Y_n(A\setminus I)) .	
	\feqn
	In addition, we have
	\beqn\label{rel2}
	f_0^v(Y_n(A\setminus I)) \leq \binom{n}{i} f_0^v([i]^{n-i}) f_0^v(\cals_i),
	\feqn
	since we may decompose each $w\in F_0^v([n]^n)$ into two words one of which is of length $i$ with $i$ distinct letters that avoids $v$ (the number of such words is $f_0^v(\cals_i)$) and the other is in $[i]^{n-i}$ avoiding $v$ (the number of such words is $f_0^v([i]^{n-i})$). These two words can be combined in $\binom{n}{i}$ ways. Hence, by setting $A=[n]$ in \eqref{rel1} and \eqref{rel2}, we arrive at
	\beqn \label{sdisc}
	f_0^v([n]^n)  \leq \sum_{i=1}^n \binom{n}{i}^2  f_0^v([i]^{n-i}) f_0^v(\cals_i).
	\feqn
	
	\section*{Acknowledgement}
	
	R.R. would like to thank Alex Roitershtein for many fruitful conversations on the pattern avoidance and occurrence.  We also would to thank Zachary Hunter for pointing out a gap in the proof of Theorem \ref{BM_theorem} in an earlier draft of this paper.

	
\end{document}